\documentclass{amsart}
\usepackage[utf8]{inputenc}
\usepackage{amsmath}
\usepackage{amsfonts}
\usepackage{amssymb}
\usepackage{amsthm}
\usepackage{graphicx}
\usepackage{enumerate}

\title[weak-strong uniqueness for NS/AC system]{Weak-Strong Uniqueness for Navier-Stokes/Allen-Cahn system}
\author{Radim Ho\v sek}
\address{Radim Ho\v sek, Department of Mathematics, University of West Bohemia, Univerzitn\'i 8, 306 14 Plzen, Czech Republic,}
\email{hosek@math.cas.cz}
\author{V\'aclav M\'acha}
\address{V\'aclav M\'acha, Institute of Mathematics of the Czech Academy of Sciences, \v Zitn\'a 25, 115 67 Praha 1, Czech Republic}
\email{macha@math.cas.cz}
\thanks{We would like to express our thanks to Amy Novick-Cohen for a fruitfull discussion. Research of V. M. was supported by GA\v CR project GA17-01747S and RVO: 67985840.}
\newcommand\dd{\, \mathrm{d}}

\newcommand\ess{\mathrm{ess}}
\newcommand\en{\textbf{n}}
\newcommand\vu{\textbf{u}}
\newcommand\vU{\textbf{U}}
\newcommand\vv{\textbf{v}}

\newcommand{\di}{\mathrm{div}_x}
\newcommand{\diver}{\mathrm{div}}
\newcommand{\intQT}{\int\limits_0^t \!\!\!\int_\Omega}
\newcommand{\intQtau}{\int\limits_0^\tau \!\!\!\int_\Omega}

\newtheorem{thm}{Theorem}[section]
\newtheorem{lemma}[thm]{Lemma}
\newtheorem{prop}[thm]{Proposition}

\theoremstyle{definition}

\begin{document}

\numberwithin{equation}{section}
\begin{abstract}
The coupled Navier-Stokes/Allen-Cahn system is a simple model to describe phase separation in two-component systems interacting with an incompressible fluid flow. We demonstrate the \emph{weak-strong uniqueness} result for this system in a bounded domain in three spatial dimensions which implies that when a strong solution exists then a weak solution emanating from the same data coincides with the strong solution on its whole life-span. The proof of given assertion relies on a form of a relative entropy method.
\end{abstract}
\maketitle

\section{Introduction}

Given a bounded Lipschitz domain $\Omega\subset \mathbb R^3$ and a time $T>0$, let us consider the following Navier-Stokes/Allen-Cahn system
\begin{eqnarray}
\label{EQmom}
\vu_t + \di (\vu \otimes \vu) + \nabla_x p &=& \di \mathbb{S}(\nabla_x \vu) - \varepsilon \di (\nabla_x c \otimes \nabla_x c),\\
\label{EQdivu}
\di \vu &=& 0,\\
\label{EQconc}
c_t + \vu \cdot \nabla_x c &=& \varepsilon \Delta_x c - \frac{1}{\varepsilon} F'(c).
\end{eqnarray}
on $Q_T:=(0,T)\times \Omega$ in conjection with the Dirichlet boundary value for velocity, i.e.
\begin{equation}\label{BCvu}
\vu |_{\partial \Omega} = 0,
\end{equation}
and Neumann boundary condition for concentration
\begin{equation}\label{BCgc}
\nabla_x c \cdot \mathbf{n}|_{\partial \Omega} = 0,
\end{equation}
which emanates from the initial conditions
\begin{equation}\label{ICvu}
\vu(0,\cdot) = \vu_0(\cdot), \qquad  c(0,\cdot) = c_0(\cdot)
\end{equation}

The system has three unknowns $\vu:Q_T\mapsto \mathbb R^3,\ p:Q_T\mapsto \mathbb R$ and $c:Q_T\mapsto \mathbb R$ which represent velocity, pressure and concentration respectively. 
Here, $\mathbb S$, the stress tensor, and $F$, the energy density,  are prescribed functions satisfying assumptions outlined in Section \ref{HAKR}.

The existence of a weak solution  to the above system was claimed (without a proof) in \cite{Zhao} assuming $\vu_0\in W^{1,2}_{0,\diver}$ and $c_0\in W^{2,2}_n$. A proof may readily be constructed as in \cite{linliu}. The uniqueness of a weak solution is still an open problem.

On the other hand, strong solution is regular enough to ensure its uniqueness. The existence of a strong solution emanating from $\vu_0\in W^{1,2}_{0,\diver}(\Omega)$, $c\in W^{2,2}_n$ on a short-time interval can also be proven as in \cite{linliu}.

The precise form of the above-mentioned existence results are stated in Section~\ref{HAKR}.

The main aim of this paper is to prove the weak-strong uniqueness for the above system, namely, if a (unique) strong solution exists, all weak solutions emanating from the same initial condition must be equal to the strong one. 

We introduce a relative entropy functional which measures distance between a weak solution and an arbitrary sufficiently smooth function, and demonstrate that the relative entropy functional satisfies a relative entropy inequality which allows us to conclude that the distance between any weak solution and the strong solution is zero so long as the strong solution exists. The relative entropy method was apparently first introduced by Dafermos \cite{dafer}. The weak-strong uniqueness property was proven for the compressible isentropic fluids using this method by Germain \cite{germain} and for the full Navier-Stokes-Fourier by Feireisl and Novotn\'y \cite{FeNo}. To the best of our knowledge, a relative entropy functional for the Navier-Stokes/Cahn-Hillard system is presented here for the first time.

Relative entropy functional provides a means of distance between a weak solution of a given problem and a sufficiently smooth functions. In our case, these functions will be a strong solution of the same problem. Another use of the relative entropy method is proving the singular limits; in such case, the smooth functions would be a solution of a target system (see e.g. \cite{BrKrMa}, \cite{DuNe}, \cite{FeKlNoZa} and many others). The relative entropy functional for the Navier-Stokes/Allen-Cahn system is introduced by relation \eqref{REfunctional}.

The paper is organized as follows. In the next section we state hypotheses and we recall some known results. The relative entropy functional is defined and the relative entropy inequality is derived in Section \ref{sec:REI}. The last Section contains our main claim and its proof.

\subsection{Notation}
Standard Lebesque, Sobolev and Bochner spaces are denoted by $L^p(\Omega)$, $W^{k,p}(\Omega)$ and $L^p(0,T,X)$ respectively. We use this notation for both real- and vector-valued function. Further, we introduce the following notation
\begin{equation*}
L^2_{div}(\Omega) : = \overline{ \{ \phi \in C^{\infty}_0(\Omega), \di \phi = 0 \} }^{\|.\|_{L^2}},
\end{equation*}
\begin{equation*}
W^{1,2}_{0,\diver}(\Omega) : = W^{1,2}_0(\Omega) \cap L^2_{\diver}(\Omega), 
\end{equation*}
\begin{equation*}
W^{s,2}_n(\Omega):=  \overline{ \{ \phi \in C^{\infty}(\Omega), \nabla \phi  \cdot \en |_{\partial \Omega}= 0 \} }^{\|.\|_{W^{s,2}}}, \qquad s\in \mathbb N.
\end{equation*}

\section{Hypothesis and known results}
\label{HAKR}
The stress tensor is assumed to satisfy the standard linear constitutive relation,	

\begin{equation*}
\mathbb{S}(\nabla_x \vu) = \frac{\nu}{2} (\nabla_x \vu + \nabla_x^T \vu),\qquad \nu>0.
\end{equation*}
$F$ is taken to be a double-well potential with two local minimizers $y_1,y_2$ which satisfies
\begin{equation}\label{eF}
F \in C^{1,1}[f_1,f_2], \qquad  -\infty<f_1 <y_1 < y_2 < f_2<\infty.
\end{equation}
Further, we assume that the initial condition $c_0$ satisfies
\begin{equation}\label{ICc}
[\ess \inf_\Omega c_0,\ess\sup_\Omega c_0] \subseteq [f_1,f_2].
\end{equation}

\subsection{Energy balance}
Proceeding formally, we multiply  (\ref{EQmom}) by $\vu$ and (\ref{EQconc}) by $(c_t + \vu \cdot \nabla_x c)$, then integrate over space to get

\begin{equation*}
\frac{1}{2} \frac{\dd}{\dd t} \int_\Omega |\vu|^2 \dd x + \nu \int_\Omega \mathbb{S}(\nabla_x \vu):\nabla_x \vu \dd x= \varepsilon\int_\Omega (\nabla_x c \otimes \nabla_x c) : \nabla_x \vu \dd x,
\end{equation*}
and
\begin{equation*}
-\int_\Omega |c_t + \vu \cdot \nabla_x c|^2 \dd x= \frac{\varepsilon}{2} \frac{\dd}{\dd t} \int_\Omega |\nabla_x c|^2 \dd x + \frac{1}{\varepsilon} \frac{\dd}{\dd t} \int_\Omega F(c) \dd x - \varepsilon \int_\Omega (\nabla_x c \otimes \nabla_x c) : \nabla_x \vu \dd x.
\end{equation*}
Combining the above equalities yield the following dissipation equation
\begin{equation}\label{energy_diss}
\frac{\dd}{\dd t} \int_\Omega \left( \frac{1}{2} |\vu|^2 + \frac{\varepsilon}{2} |\nabla_x c|^2 + \frac{1}{\varepsilon} F(c)  \right) \dd x + \int_\Omega \left( \mathbb{S}(\nabla_x \vu):\nabla_x \vu + |c_t + \vu \cdot \nabla_x c|^2 \right) \dd x = 0.
\end{equation}

\subsection{Weak formulation}

We say that the couple $(\vu, c)$

\begin{equation}\label{wf_vu}
\vu \in L^\infty\left(0,T;L^2_{div}(\Omega)\right) \cap L^2\left(0,T; W^{1,2}_{0,div}(\Omega)\right),
\end{equation}
\begin{equation}\label{wf_c}
c \in L^\infty\left(0,T;W^{1,2}_n(\Omega)\right) \cap L^2\left(0,T; W^{2,2}_n(\Omega)\right),
\end{equation}
is a weak solution to the Navier-Stokes/Allen-Cahn system, if it satisfies
\begin{equation}\label{WF_mom}
\begin{split}
\int_\Omega \vu(t) \cdot \vv(t) \dd x - \int_{\Omega} \vu_0 \cdot \vv(0) - \intQT \vu\cdot \vv_t \dd x\dd t - \intQT (\vu \otimes \vu):\nabla_x \vv \dd x\dd t \\ +\intQT \mathbb{S}(\nabla_x \vu):\nabla_x \vv \dd x \dd t = \varepsilon \intQT (\nabla_x c \otimes \nabla_x c): \nabla_x\vv \dd x \dd t,
\end{split}
\end{equation}
for all $\vv \in C^\infty(Q_T)$ such that $\vv = 0$ on $(0,T)\times\partial \Omega$ and $\di \vv = 0$ for a.e. $t \in [0,T]$,

\begin{equation}\label{WF_conc}
c_t + \vu \cdot \nabla_x c = \varepsilon \Delta_x c - \frac{1}{\varepsilon} F'(c),
\end{equation}
almost everywhere in $\Omega \times [0,T]$, as well as the energy inequality

\begin{equation}\label{en_ineq}
\begin{split}
&  \int_\Omega \left( \frac{1}{2} |u(t)|^2 + \frac{\varepsilon}{2} |\nabla_x c|^2 + \frac{1}{\varepsilon} F(c) \right)(x,t) \dd x \\ & \quad + \intQT  \left( \mathbb{S}(\nabla_x \vu) : \nabla_x \vu + |c_t + \vu \cdot \nabla_x c|^2 \right) \dd x \dd t \\ & \quad \leq \int_\Omega \left( \frac{1}{2} |u_0|^2 + \frac{\varepsilon}{2} |\nabla_x c_0|^2 + \frac{1}{\varepsilon} F(c_0) \right) \dd x, 
 \end{split}
\end{equation}
for a.e. $t \in [0,T]$. 

We emphasise that the weak solution is supposed to fulfill the energy inequality rather than equality. The energy dissipation \eqref{energy_diss} was derived just formally and it holds only for sufficiently smooth solutions. Unfortunately, \eqref{wf_vu} does not allow to multiply \eqref{EQmom} by $\vu$.

Since we are considering an incompressible fluid, the pressure does not appear in the weak formulation. However, it can be reconstructed by standard techniques, see for example \cite{sohr}.

\subsection{Existence}

\begin{thm}[\cite{Zhao}, Theorem 2.1]
Let $\vu_0 \in L^2_{div}(\Omega)$  and $c_0 \in W^{1,2}_n(\Omega)$ satisfy \eqref{ICc} and let \eqref{eF} hold. Then the system (\ref{EQmom}--\ref{ICvu}) possesses a global weak solution 
\begin{multline*}(\vu,c)\in \left(L^{2}(0,T,W^{1,2}_{0,\diver}(\Omega))\cap L^\infty(0,T,L^2(\Omega))\right)\\
\times \left(L^\infty(0,T,W^{1,2}(\Omega))\cap L^2(0,T,W^{2,2}(\Omega))\right).\end{multline*}
Moreover, if $\vu_0 \in W^{1,2}_{0,div}(\Omega)$ and $c_0 \in W^{2,2}_n(\Omega)$, then there exists $T^\star > 0$ such that (\ref{EQmom}--\ref{ICvu}) possesses a unique strong solution $(\vU,C)$ in $[0,T^\star)$ such that 

\begin{equation}\label{integr_vU}
\vU \in L^\infty\left(0,T;W^{1,2}_{0,div}(\Omega)\right) \cap L^2\left(0,T; W^{2,2}_{0,div}(\Omega)\right),
\end{equation}
\begin{equation}\label{integr_C}
C \in L^\infty\left(0,T;W^{2,2}_n(\Omega)\right) \cap L^2\left(0,T; W^{3,2}_n(\Omega)\right)
\end{equation}
holds for $T < T^\star$.
\end{thm}

The uniqueness of the strong solution is thus known. As mentioned before, our aim is to show uniqueness of the strong solution in the class of weak solutions. The integrability properties, which allow to use the relative entropy inequality, are
\begin{equation}\label{integr_vuc}
\vU \in L^2(0,T; L^\infty(\Omega)^3), \quad \nabla_x C \in L^2(0,T;L^\infty(\Omega)).
\end{equation}
which follows from (\ref{integr_vU}--\ref{integr_C}) and the Sobolev embedding theorem  in 3 dimensions. As (\ref{integr_vuc}) does not hold in general for weak solutions, we are not able to show the uniqueness of weak solutions.
  
\subsection{Weak maximum principle on the concentration}

Both strong and weak solutions of the Allen-Cahn equation with convection satisfy the maximum (and minimum) principle as specified in the following Proposition.

\begin{prop}[Weak maximum principle]
Let $(\vu,c)$ satisfy \eqref{wf_vu}, \eqref{wf_c}) and \eqref{WF_conc} almost everywhere in $Q_T$  with boundary condition \eqref{BCgc} and initial condition $c(0,\cdot) = c_0(\cdot)$ fulfilling \eqref{ICc}. Then for a.e. $(x,t) \in Q_T$,

\begin{equation*}
c(x,t) \in [m,M] := \mathrm{co}\left\{\ess \inf_{x\in \Omega} c_0, \ess \sup_{x \in \Omega} c_0,y_1,y_2\right\},
 \end{equation*} 
 where $y_1, y_2$ are the local minimizers of $F$. 
\end{prop}

\begin{proof}
We show the minimum principle only, the latter inequality is analogous. By assumptions on $c_0$ and $F$, we have $m >-\infty$ and we   test (\ref{WF_conc}) by $(c-m)^-$ to get

\begin{equation*}
\frac{1}{2} \int_\Omega \frac{\dd}{\dd t} \left((c-m)^-\right)^2 \dd x+ \varepsilon \int_\Omega |\nabla_x (c-m)^-|^2 \dd x + \frac{1}{\varepsilon} \int_\Omega F'(c) (c-m)^- \dd x = 0.
 \end{equation*}
As the second and third integrals in the above equation have positive sign (monotonicity of $F$), we recover that $((c-m)^-)^2 = 0$ for a.a. $t\in [0,T]$ and, consequently, $c(x,t) \geq m$ for a.e. $(x,t) \in Q_T$. 
\end{proof}

\section{Relative entropy}\label{sec:REI}
\subsection{Relative entropy functional}

The particular form of the relative entropy functional reads
\begin{equation}\label{REfunctional}
\begin{split}
\mathcal{E}(\vu,c | \vU,C)(t) = \int_\Omega \left( \frac{1}{2}|\vu  - \vU|^2 + \frac{\varepsilon}{2}| \nabla_x (c - C)|^2  \right)(x,t) \dd x. 
\end{split}
\end{equation}

To prove that the relative entropy functional indeed measures a distance of two solutions, we need to show that it does not give zero value for $c \neq C$. This is a direct consequence of the following lemma.

\begin{lemma}\label{L:poinc} Let $\Omega$ be a bounded Lipschitz domain $T>0$, let $F$ satisfy \eqref{eF} and $c_0$ satisfy \eqref{ICc}. Let $\vu_1,\ \vu_2\in L^{2}(0,T,W^{1,2}_{0,\diver}(\Omega))$, $c_1,\ c_2$ satisfy \eqref{wf_c}. There exists $K$ independent on $\vu_1,\ c_1,\ \vu_2$ and $c_2$, such that if $c_1$ and $c_2$ are solutions to \eqref{WF_conc} with corresponding velocities emanating from $c_0$, then
\begin{equation*}
\int_\Omega (c_1-c_2)^2(t) \dd x \leq K  \intQT\left( (\nabla_x (c_1-c_2))^2 + (\vu_1 - \vu_2)^2 \right) \dd x, 
\end{equation*}
holds for a.e. $t \in [0,T]$. 
\end{lemma}

As a consequence, if $\mathcal{E}(\vu,c|\vU,C)(t) = 0$ for a.e. $t \in [0,T]$, then $c = C$ a.e in $\Omega \times [0,T]$. 

\begin{proof}
We take the difference of the equations (\ref{WF_conc}) for the two solutions, test by $(c_1-c_2)$ and integrate over $(0,\tau)\subset(0,T)$ to obtain
\begin{equation}\label{poinc_pf1}
\begin{split}
& \int_\Omega \frac{1}{2}(c_1-c_2)^2(\tau) \dd x = - \int_0^\tau\int_\Omega (c_1-c_2) (\vu_1 \cdot \nabla_x c_1 - \vu_2 \cdot \nabla_x c_2) \dd x \dd t \\ & \quad + \varepsilon \int_0^\tau\int_\Omega |\nabla_x (c_1-c_2)|^2 \dd x \dd t + \frac{1}{\varepsilon} \int_0^\tau\int_\Omega (F'(c_1)-F'(c_2))(c_1-c_2) \dd x \dd t.
\end{split}
\end{equation}

First, we focus on the difference of the convective terms. Integration by parts together with boundary conditions and solenoidality of $\vu_1, \vu_2$ yield
\begin{multline}\label{poinc_pf2}
- \int_0^\tau\int_\Omega (c_1-c_2) (\vu_1 \cdot \nabla_x c_1 - \vu_2 \cdot \nabla_x c_2) \dd x \dd t \\  \quad   =  \int_0^\tau\int_\Omega c_1\vu_2\cdot \nabla_x c_2 \dd x \dd t +   \int_0^\tau\int_\Omega c_2\vu_1\cdot \nabla_x c_1 \dd x \dd t =  - \int_0^\tau\int_\Omega c_1 (\vu_1 -\vu_2 ) \cdot \nabla_x c_2 \dd x \dd t  \\  \quad = - \int_0^\tau\int_\Omega (c_1-c_2) (\vu_1 -\vu_2 ) \cdot \nabla_x c_2 \dd x \dd t   =  \int_0^\tau\int_\Omega c_2 (\vu_1 -\vu_2 ) \cdot \nabla_x (c_1-c_2) \dd x \dd t.
\end{multline}
Since $F'$ is Lipschitz, the last term on the right hand side of \eqref{poinc_pf1} can be estimated as
\begin{equation}\label{lipschitz_treatment}
\frac{1}{\varepsilon} \int_0^\tau\int_\Omega (F'(c_1)-F'(c_2))(c_1-c_2) \dd x \dd t \leq \frac{L_{F'}}{\varepsilon} \int_0^\tau\int_\Omega  (c_1-c_2)^2 \dd x \dd t.
\end{equation}
 
We use \eqref{poinc_pf1}, \eqref{poinc_pf2}, \eqref{lipschitz_treatment} and the Cauchy-Schwarz inequality in order to deduce
\begin{equation*}
\begin{split}
& \int_\Omega \frac{1}{2}|c_1-c_2|^2(\tau) \dd x \leq  \int_0^\tau \left( (\ess \sup_{\Omega} \frac{1}{2}|c_2|^2 + \varepsilon) \int_\Omega |\nabla_x (c_1-c_2)|^2 \dd x \right) \dd t \\ & \quad + \frac{1}{2} \int_0^\tau\int_\Omega  |\vu_1 -  \vu_2|^2    \dd x \dd t + \frac{L_{F'}}{\varepsilon} \int_0^\tau\int_\Omega |c_1-c_2|^2 \dd x \dd t.
\end{split}
\end{equation*}
The integrability properties of $(\vu_i,c_i)$ allow to apply the Gronwall inequality which concludes the proof.
\end{proof}

\subsection{Relative entropy inequality}
\begin{prop}
Let $(\vU, C)$ be a strong solution to (\ref{EQmom}--\ref{ICvu}). Then the relative entropy functional satisfies the following relative entropy inequality

\begin{equation}\label{REI}
\begin{split}
& \mathcal E(u,c|U,C)(t) + \intQT \mathbb{S}(\nabla(\vu - \vU)):\nabla_x (\vu - \vU) \dd x \dd t \\& \quad + \intQT |(c_t + \vu \cdot \nabla_x c) - (C_t + \vU \cdot \nabla_x C)|^2  \dd x \dd t  \\ & \quad - \int_\Omega \left( \frac{1}{2}|\vu_0 - \vU_0|^2 + \frac{\varepsilon}{2}| \nabla_x (c_0 - C_0)|^2 \right) \dd x  
 \\ &  \leq \intQT  ((\vu - \vU)  \otimes \vU): \nabla_x (\vu- \vU) \dd x \dd t \\ & \quad
+ \varepsilon \intQT \Delta_x(c-C) \vU \cdot \nabla_x(c-C) \dd x \dd t  \\ & \quad  + \varepsilon \intQT \nabla_x C \otimes \nabla_x (c-C) : \nabla_x (\vu-\vU) \dd x \dd t  \\ & \quad + \varepsilon \intQT \nabla_x (c-C) \otimes \nabla_x C : \nabla_x (\vu-\vU) \dd x \dd t  \\ & \quad + \varepsilon \intQT \Delta_x (c-C)   (\vu-\vU) \cdot  \nabla_x C \dd x \dd t \\ & \quad 
-\frac{1}{\varepsilon} \intQT ( F'(c) - F'(C))\times \left((c_t + \vu \cdot \nabla_x c) - (C_t + \vU \cdot \nabla_x C) \right) \dd x \dd t. 
\end{split}
\end{equation}
\end{prop}
The remainder of Section \ref{sec:REI} is devoted to a proof of (\ref{REI}). 

\subsection{Initial estimates}
In order to obtain (\ref{REI}), we note the following:

\begin{itemize}
\item the energy inequality (\ref{en_ineq}) which holds for weak solutions,
\begin{equation}\label{REIpf01}
\begin{split}
& \frac{1}{2} \int_\Omega |\vu|^2(x,t) \dd x + \frac{\varepsilon}{2} \int_\Omega |\nabla_x c|^2(x,t) \dd x + \frac{1}{\varepsilon} \int_\Omega F(c)(x,t) \dd x \\ & \quad + \intQT \mathbb{S}(\nabla_x \vu):\nabla_x \vu \dd x\dd t + \intQT |c_t + \vu \cdot \nabla_x c|^2 \dd x \dd t \\ & \quad \leq \frac{1}{2} \int_\Omega |\vu_0|^2 \dd x + \frac{\varepsilon}{2} \int_\Omega |\nabla_x c_0|^2 \dd x + \frac{1}{\varepsilon} \int_\Omega F(c_0) \dd x,
\end{split}
\end{equation}
\item the weak formulation of the momentum equation (\ref{WF_mom}) with $\vv  = -\vU$:
\begin{equation}\label{REIpf02}
\begin{split}
& - \int_\Omega (\vu\cdot \vU)(x,t) \dd x + \int_\Omega \vu_0 \cdot \vU_0 \dd x + \intQT \vu  \cdot \vU_t \dd x \dd t \\ & \quad + \intQT (\vu \otimes \vu): \nabla_x \vU \dd x \dd t - \intQT \mathbb{S}(\nabla_x \vu): \nabla_x \vU \dd x \dd t \\ & \quad = - \varepsilon \intQT (\nabla_x c \otimes \nabla_x c): \nabla_x \vU \dd x \dd t.
\end{split}
\end{equation}
Note, that the regularity of $\vU$ is sufficient. Indeed, $\vU$ can be approximated by smooth selenoidal functions which are allowed to be test functions in \eqref{WF_mom}. As all terms in \eqref{REIpf02} have sense, the convergence of corresponding integrals is a standard matter.
\item  The equation for concentration (\ref{WF_conc}) multiplied by $-(C_t + \vU\cdot \nabla_x C) $:
\begin{equation}\label{REIpf03}
\begin{split}
& - \intQT (c_t + \vu \cdot \nabla_x c)(C_t + \vU \cdot \nabla_x C) \dd x \dd t  \\ & \quad = \varepsilon \intQT \nabla_x c \cdot\nabla_x C_t \dd x \dd t + \varepsilon \intQT \nabla_x c \cdot \nabla_x (\vU \nabla_x C) \dd x \dd t \\ & \quad + \frac{1}{\varepsilon} \intQT F'(c) (C_t + \vU \cdot \nabla_x C) \dd x \dd t.
\end{split}
\end{equation}
\item The momentum equation (\ref{EQmom}) for the strong solution $(\vU,C)$ multiplied by $(\vU - \vu)$:
\begin{equation}\label{REIpf04}
\begin{split}
& \frac{1}{2} \int_\Omega |\vU|^2(x,t) \dd x - \frac{1}{2} \int_\Omega |\vU_0|^2 \dd x - \intQT \vU_t \cdot  \vu \dd x \dd t  + \intQT (\vU \otimes \vU) : \nabla_x \vu \dd x \dd t \\ & \quad + \intQT \mathbb{S}(\nabla_x \vU):\nabla_x (\vU - \vu)  \dd x \dd t - \varepsilon \intQT (\nabla_x C \otimes \nabla_x C): \nabla_x (\vU - \vu) \dd x \dd t = 0.
\end{split}
\end{equation}
\item The concentration equation (\ref{EQconc}) for strong solutions $(\vU,C)$ multiplied by $(C_t + \vU\cdot \nabla_x C) - (c_t + \vu \cdot \nabla_x c )$:
\begin{equation}\label{REIpf05}
\begin{split}
& \intQT |C_t + \vU \cdot \nabla_x C|^2 \dd x \dd t - \intQT (C_t + \vU \cdot \nabla_x C)(c_t + \vu \cdot \nabla_x c) \dd x \dd t \\& \quad = - \frac{\varepsilon}{2} \int_\Omega |\nabla_x C |^2 (x,t) \dd x + \frac{\varepsilon}{2} \int_\Omega |\nabla_x C_0|^2 \dd x  \\ & \quad  + \varepsilon \intQT \Delta_x C 
(\vU \cdot \nabla_x C) \dd x \dd t - \varepsilon \intQT \Delta_x C  ( c_t + \vu \cdot \nabla_x c) \dd x \dd t \\ & \quad   - \frac{1}{\varepsilon} \intQT F'(C)(C_t +  U \cdot \nabla_x C) \dd x  \dd t +  \frac{1}{\varepsilon} \intQT F'(C)( c_t + \vu\cdot \nabla_x c) \dd x \dd t.  
\end{split}
\end{equation}
\end{itemize}

\subsection{Formation of the left hand side of (\ref{REI})} 
We sum the relations (\ref{REIpf01}--\ref{REIpf05}). With the help of integration by parts we obtain

\begin{equation}\label{gather}
\begin{split}
& \int_\Omega \left( \frac{1}{2}|\vu - \vU|^2 + \frac{\varepsilon}{2}| \nabla_x (c - C)|^2 \right)(x,t) \dd x \\ & \quad + \intQT \mathbb{S}(\nabla(\vU - \vu)):\nabla_x (\vU - \vu) \dd x \dd t \\& \quad + \intQT |(c_t + \vu \cdot \nabla_x c) - (C_t + \vU \cdot \nabla_x C)|^2  \dd x \dd t  \\ & \quad - \int_\Omega \left( \frac{1}{2}|\vu_0 - \vU_0|^2 + \frac{\varepsilon}{2}| \nabla_x (c_0 - C_0)|^2 \right) \dd x   \leq \mathcal{R}. \end{split}
\end{equation}
The left hand side of (\ref{gather}) is the difference of the relative entropy functionals plus non-negative dissipation terms, while all the other terms were put to the right hand side $\mathcal{R}$. We can split $\mathcal{R}$ into three parts, 
\begin{equation*}
\mathcal{R} = \mathcal{R}_{conv} + \mathcal{R}_{\varepsilon} + \mathcal{R}_{F}.
\end{equation*}
$\mathcal{R}_{conv}$ contains the remaining convective terms, namely
\begin{equation}\label{Rconv}
\mathcal{R}_{conv} =  - \intQT (\vU \otimes \vU) : \nabla_x \vu \dd x \dd t  - \intQT (\vu \otimes \vu): \nabla_x \vU \dd x \dd t. 	
\end{equation} 
$\mathcal{R}_{\varepsilon}$ is a sum of all remaining terms, led by $\varepsilon$, i.e.,
\begin{equation*}
\begin{split}
& \mathcal{R}_\varepsilon = - \varepsilon \intQT (\nabla_x c \otimes \nabla_x c): \nabla_x \vU \dd x \dd t + \varepsilon \intQT \nabla_x c \cdot \nabla_x (\vU \cdot \nabla_x C) \dd x \dd t  \\ & \quad  +\varepsilon \intQT (\nabla_x C \otimes \nabla_x C): \nabla_x (\vU - \vu) \dd x \dd t + \varepsilon \intQT \Delta_x C  (\vU \cdot \nabla_x C) \dd x \dd t \\ & \quad - \varepsilon \intQT \Delta_x C (\vu \cdot \nabla_x c) .
\end{split} 
\end{equation*}
Finally, $\mathcal{R}_{F}$ is a sum of all terms led by $1/{\varepsilon}$:
\begin{equation}\label{RF}
\begin{split}
& \varepsilon \mathcal{R}_{F} =  - \int_\Omega F(c)(x,t) \dd x  + \int_\Omega F(c_0) \dd x + \intQT F'(c) (C_t + \vU \cdot \nabla_x C) \dd x \dd t \\ & \quad  -   \intQT F'(C)(C_t + \vU \cdot \nabla_x C) \dd x \dd t +  \intQT F'(C) (c_t + \vu \cdot \nabla_x c) \dd x \dd t.
\end{split}
\end{equation}

\subsection{Reformulation of the right hand side}
We treat $\mathcal{R}_{conv}$ first. Using the identities
\begin{equation*}
- \int_\Omega \vu \otimes \vu : \nabla_x\vU \dd x =  \int_\Omega  \vu \otimes \vU : \nabla_x \vu \dd x, \quad \int_\Omega  (\vu-\vU)\otimes \vU : \nabla_x \vU \dd x = 0 
\end{equation*}
we can rewrite (\ref{Rconv}) into 

\begin{equation}\label{REI_RHS_conv}
\mathcal{R}_{conv} = \intQT  ((\vu - \vU)  \otimes \vU): \nabla_x (\vu- \vU) \dd x \dd t. 
\end{equation}

We simplify $\mathcal{R}_{\varepsilon}$ using integration by parts, in particular, 

\begin{equation*}
\begin{split}
& \mathcal{R}_\varepsilon = \varepsilon \intQT  \Delta_x c \vU \cdot \nabla_x c \dd x \dd t - \varepsilon \intQT  \Delta_x c \vU \cdot \nabla_x C \dd x \dd t \\& \quad  - \varepsilon \intQT  \Delta_x C \vU \cdot \nabla_x C \dd x \dd t + \varepsilon \intQT  \Delta_x C \vu \cdot \nabla_x C \dd x \dd t \\ & \quad + \varepsilon \intQT  \Delta_x C \vU \cdot \nabla_x C \dd x \dd t - \varepsilon \intQT  \Delta_x C \vu \cdot \nabla_x c \dd x \dd t.  
\end{split} 
\end{equation*}
The third and the fifth terms cancel out, the remaining four are equal to
\begin{equation}\label{Reps_2int}
\begin{split}
& \mathcal{R}_\varepsilon =   \varepsilon \intQT \Delta_x(c-C) \vU \cdot \nabla_x(c-C) \dd x \dd t   \\   & \quad 
-  \varepsilon \intQT \Delta_x C (\vu-\vU) \cdot \nabla_x(c-C) \dd x \dd t.
\end{split}
\end{equation}
The latter integral in (\ref{Reps_2int}) is not in a form that enables its estimation. We reformulate it using three integrations by parts, in particular
\begin{equation}\label{Reps_3pp}
\begin{split}
& -  \varepsilon \intQT \Delta_x C  (\vu-\vU) \cdot \nabla_x(c-C)  \dd x \dd t \\ & \quad  = \varepsilon \intQT \nabla_x C \otimes \nabla_x (c-C) : \nabla_x (\vu-\vU) \dd x \dd t \\ & \qquad +  \varepsilon \intQT \nabla_x C \otimes (\vu-\vU) : \nabla_x \otimes \nabla_x (c-C)  \dd x \dd t \\ & \quad = \varepsilon \intQT \nabla_x C \otimes \nabla_x (c-C) : \nabla_x (\vu-\vU) \dd x \dd t \\& \qquad  -  \varepsilon \intQT \nabla_x (c-C) \otimes (\vu-\vU) : \nabla_x \otimes \nabla_x C  \dd x \dd t \\ & \quad = 
\varepsilon \intQT \nabla_x C \otimes \nabla_x (c-C) : \nabla_x (\vu-\vU) \dd x \dd t \\ & \qquad + \varepsilon \intQT \nabla_x (c-C) \otimes \nabla_x C : \nabla_x (\vu-\vU) \dd x \dd t  \\ & \qquad + \varepsilon \intQT \Delta_x (c-C)   (\vu-\vU) \cdot  \nabla_x C \dd x \dd t.
\end{split}
\end{equation}
We deduce from (\ref{Reps_2int}--\ref{Reps_3pp}) that
\begin{equation}\label{REI_RHS_eps}
\begin{split}
& \mathcal{R}_\varepsilon =   \varepsilon \intQT \Delta_x(c-C) \vU \cdot \nabla_x(c-C) \dd x \dd t   + \varepsilon \intQT \nabla_x C \otimes \nabla_x (c-C) : \nabla_x (\vu-\vU) \dd x \dd t  \\ & \quad + \varepsilon \intQT \nabla_x (c-C) \otimes \nabla_x C : \nabla_x (\vu-\vU) \dd x \dd t  \\ & \quad + \varepsilon \intQT \Delta_x (c-C)   (\vu-\vU) \cdot  \nabla_x C \dd x \dd t =: \sum_{j=1}^4 \mathcal{R}_{\varepsilon,j}.
\end{split}
\end{equation}
Finally, we treat $\mathcal{R}_F$. We write all terms in (\ref{RF}) as space-time integrals to get
\begin{equation}\label{REI_RHS_F}
\begin{split}
&  \mathcal{R}_{F} =  - \frac{1}{\varepsilon} \intQT F'(c)(c_t + \vu \cdot \nabla_x c) \dd x \dd t  +  \frac{1}{\varepsilon} \intQT F'(c) (C_t + \vU \cdot \nabla_x C) \dd x \dd t \\ & \quad  - \frac{1}{\varepsilon}  \intQT F'(C) C_t \dd x \dd t + \frac{1}{\varepsilon} \intQT F'(C) (c_t + \vu \cdot \nabla_x c) \dd x \dd t \\ & \quad = 
-\frac{1}{\varepsilon} \intQT ( F'(c) - F'(C))\times \left((c_t + \vu \cdot \nabla_x c) - (C_t + \vU \cdot \nabla_x C) \right) \dd x \dd t. 
\end{split}
\end{equation}
The desired inequality follows by combination of (\ref{gather}),(\ref{REI_RHS_conv}), (\ref{REI_RHS_eps}) and (\ref{REI_RHS_F}).

\section{Weak-strong uniqueness property}

\begin{thm}
Let $\vu_0 \in W^{1,2}_{0,div}(\Omega)^3$ and $c_0 \in W^{2,2}_n(\Omega)$ fulfill \eqref{ICc} and  $(\vu,c)$ be a weak solution and $(\vU,C)$ the strong solution to Navier-Stokes/Allen-Cahn system (\ref{EQdivu}--\ref{ICc}) both emanating from the same initial data $(\vu_0,c_0)$. Then on the life span $[0,T^\star)$ of the strong solution we have $(\vu,c) = (\vU,C)$.
\end{thm}

\begin{proof}
The proof of the weak-strong uniqueness uses the Gronwall-type argument. Hence we aim at rewriting the relative entropy inequality in the form 

\begin{equation}\label{Gw_scheme}
\mathcal{E}(\tau) - \mathcal{E}_0 + D \leq \lambda D + k \int\limits_0^\tau \omega(s) \mathcal{E}(s) \dd s, 
\end{equation}

where
\begin{itemize}
\item $k>0$ is a (possibly large) constant, independent of time,
\item $D$ denotes the dissipative terms and $\lambda \in [0,1)$,
\item $\omega \in L^{1}[0,T]$ for all $T \in (0,T^\star)$,
\item (\ref{Gw_scheme}) holds for almost all $\tau \in [0,T^\star)$. 
\end{itemize}

After reaching (\ref{Gw_scheme}), one employs the Gronwall inequality to obtain the desired conclusion. Hence the whole proof reduces to showing (\ref{Gw_scheme}).

We focus on right hand side terms of (\ref{REI}). They all share similar structure: They are integrals of three factors, two of them being in a difference form. Roughly speaking, one of those differences is  $L^2(0,T;L^2(\Omega))$ integrable and is a part of the dissipation $D$. The latter difference is $L^\infty(0,T;L^2(\Omega))$ integrable and is a part of the relative entropy (\ref{REI}). Finally the third factor which is not a difference is  $L^2(0,T;L^\infty(\Omega))$ integrable. 
 
The strategy for reaching (\ref{Gw_scheme}) is to use weighted  Young's inequality with small weight to the first type factors to ensure $\lambda < 1$, while the large weight is kept with the latter. The latter integral contains two factors whose integrability properties match ideally. This is crucial to ensure that the condition on $\omega$ in the scheme (\ref{Gw_scheme}) is fulfilled. In particular, we have

\begin{equation}\label{WSU_est1}
\begin{split}
& |\mathcal{R}_{conv}| \leq \intQtau |\nabla_x (\vu - \vU)| |\vU | |\vu - \vU| \dd x \dd t  \\ & \quad \leq \frac{\delta}{2} \intQtau |\nabla_x (\vu - \vU)|^2 \dd x \dd t + \frac{1}{2\delta} \int\limits_0^\tau \left(\ess\sup_\Omega |\vU|^2 \int_\Omega |\vu - \vU|^2 \dd x \right)\dd t. 
\end{split}
\end{equation}

\begin{equation}\label{WSU_est2}
\begin{split}
& |\mathcal{R}_{\varepsilon,1}| \leq \varepsilon \intQtau |\Delta_x (c-C)||\vU| |\nabla_x (c-C)|  \dd x \dd t \\ & \quad 
\leq \frac{\delta \varepsilon}{2} \intQtau |\Delta_x (c-C)|^2 \dd x \dd t + \frac{\varepsilon}{2\delta} \int\limits_0^\tau \left(\ess\sup_\Omega |\vU|^2 \int_\Omega |\nabla_x (c-C)|^2 \dd x \right)\dd t.  
\end{split}
\end{equation}

\begin{equation}\label{WSU_est3}
\begin{split}
& |\mathcal{R}_{\varepsilon,2}| + |\mathcal{R}_{\varepsilon,3}| \leq 2\varepsilon\intQtau |\nabla_x (c-C)||\nabla_x C| |\vu-\vU| \dd x \dd t \\ & \quad 
\leq \delta \varepsilon \intQtau |\vu-\vU|^2 \dd x \dd t + \frac{\varepsilon}{\delta} \int\limits_0^\tau \left(\ess\sup_\Omega |\nabla_x C|^2 \int_\Omega |\nabla_x (c-C)|^2 \dd x \right)\dd t.
\end{split}
\end{equation}

\begin{equation}\label{WSU_est4}
\begin{split}
& |\mathcal{R}_{\varepsilon,4}| \leq \varepsilon \intQtau |\Delta_x (c-C)||\vu-\vU| |\nabla_x C|  \dd x \dd t \\ & \quad 
\leq \frac{\delta \varepsilon}{2} \intQtau |\Delta_x (c-C)|^2 \dd x \dd t + \frac{\varepsilon}{2\delta} \int\limits_0^\tau \left(\ess\sup_\Omega |\nabla_x C|^2 \int_\Omega |\vu-\vU|^2 \dd x \right)\dd t,
\end{split}
\end{equation}

and finally, using the Lipschitz property of the function $F'$, 

\begin{equation}\label{WSU_est5}
\begin{split}
& \mathcal{R}_F \leq  \frac{1}{\varepsilon} \intQtau |F'(c)-F'(C)| |(c_t + \vu \cdot \nabla_x c) - (C_t + \vU \cdot \nabla_x C)| \dd x \dd t \\ & \quad 
\leq \frac{\delta}{2\varepsilon} \intQtau |(c_t + \vu \cdot \nabla_x c) - (C_t + \vU \cdot \nabla_x C)|^2  \dd x \dd t + \frac{L_{F'}}{2\delta \varepsilon} \intQtau |c-C|^2 \dd x  \dd t.  
\end{split}
\end{equation}

In the last-but-one step, we use the equations (\ref{EQconc}), (\ref{WF_conc}) to treat the term 

\begin{equation}\label{estLaplace}
\begin{split}
& \frac{\delta \varepsilon}{2} \intQtau |\Delta_x (c-C)|^2 \dd x\dd t \\  & \quad \leq \frac{\delta}{\varepsilon}\intQtau \left|\varepsilon \Delta_x (c-C) - \frac{1}{\varepsilon}(F'(c)-F'(C))\right|^2 \dd x\dd t + \frac{\delta}{\varepsilon} \intQtau \left|\frac{1}{\varepsilon}(F'(c)-F'(C))\right|^2 \dd x\dd t \\ & \quad \leq \frac{\delta}{\varepsilon}  \intQtau |(c_t + \vu \cdot \nabla_x c) - (C_t + \vU \cdot \nabla_x C)|^2 \dd x \dd t + \frac{\delta L^2_{F'}}{\varepsilon^3} \intQtau |c-C|^2.
\end{split}
\end{equation}

Finally, we collect all terms on the right hand sides of (\ref{WSU_est1}--\ref{WSU_est5}) and apply (\ref{estLaplace}) and also Lemma \ref{L:poinc}. Then, clearly, with a proper choice of $\delta$, one gets the inequality of the Gronwall type (\ref{Gw_scheme}). Applying the Gronwall's inequality yields $\mathcal{E}(\vu, c | \vU, C)(t) = 0$ for almost all $t \in [0,T^\star)$, which concludes the proof. 
\end{proof}

\bibliographystyle{abbrv}
\bibliography{wsu_nsac_bib}

\end{document}